\documentclass[a4paper]{amsart}
\usepackage[utf8]{inputenc}
\usepackage[english]{babel}
\usepackage{amsmath}
\usepackage{amsthm}
\usepackage{amsfonts}
\usepackage{amssymb}
\usepackage{dsfont}

\newtheorem{thm}{Theorem}[]

\renewcommand{\k}{\ensuremath{\mathds{k}}}
\newcommand{\R}{\ensuremath{\mathds{R}}}
\newcommand{\C}{\ensuremath{\mathds{C}}}
\newcommand{\Q}{\ensuremath{\mathds{Q}}}
\newcommand{\U}{\ensuremath{\mathbf{U}}}

\DeclareMathOperator{\Sym}{Sym}
\DeclareMathOperator{\GL}{GL}

\DeclareMathOperator{\re}{Re}
\DeclareMathOperator{\im}{Im}

\title{Unitary Groups as Stabilizers of Orbits}
\author{Erik Friese}
\address{Universität Rostock, Institut für Mathematik, Ulmenstr. 69, Haus 3, 18057 Rostock, Germany}
\email{erik.friese@uni-rostock.de}
\thanks{The author is partially supported by DFG-grant SCHU 1503/6-1}


\begin{document}
\begin{abstract}
We show that a finite unitary group which has orbits spanning the whole space is necessarily the setwise stabilizer of a certain orbit. 
\end{abstract}
\maketitle

I. M. Isaacs has shown that any finite matrix group $G \subset \GL(n,\k)$ over an infinite field $\k$ is the setwise stabilizer of a finite subset $X \subset \k^n$ \cite{Isa77}.
If $G$ acts absolutely irreducible on $\k^n$, it is even possible to take a single orbit of $G$ for $X$.

In general $X$ cannot be chosen as an orbit of $G$, for example if no orbit of $G$ linearly spans $\k^n$. But even if there are orbits of $G$ spanning the space, the stabilizer of any $G$-orbit may be strictly larger than $G$. Consider for example the orthogonal group $G$ generated by the 90 degree rotation of the plane. 
The orbit of any nonzero vector forms a square which has reflection symmetries not contained in $G$. So the setwise stabilizer of any orbit is strictly larger than $G$, even in the orthogonal group of the plane.

In fact, there are exceptional isomorphism types of finite groups which can never occur as (orthogonal or linear) setwise stabilizer of one of their orbits. In the orthogonal case these groups were fully classified by L. Babai \cite{Bab77}.
In the linear case the classification was recently done for $\k = \R$ and $\k = \Q$ in joint work with F. Ladisch \cite{FrLa16}.
The present note deals with the complex unitary case where no such exceptional isomorphism types arise. In fact, we have an even stronger result.

\begin{thm} \label{thm:main}
Let $G \leq \U(\C^n)$ be a finite unitary group such that some orbit of $G$ spans $\C^n$ (over $\C$). Then there is an open and dense subset of elements $x \in \C^n$ such that $G = \U(Gx)$. In particular, $G$ is the setwise stabilizer of one of its orbits.
\end{thm}

Here $\U(\C^n)$ denotes the group of all complex unitary $(n \times n)$-matrices, and $\U(X) = \{ A \in \U(\C^n) : AX = X \}$ denotes the stabilizer of $X \subseteq \C^n$ in $\U(\C^n)$.
In the following, we regard $\C^n$ as a vector space over $\R$ and equip $\C^n$ with the real Zariski topology.
Explicitly, the closed subsets of $\C^n$ are the common zero sets of real polynomials $f \in \R[X_1,Y_1,\dots,X_n,Y_n]$, where
\[ f(x) = f(\re(x_1),\im(x_1),\dots,\re(x_n),\im(x_n)) \text{ for } x \in \C^n. \]
For a general treatment of the Zariski topology over arbitrary fields, we refer to \cite{DuFo04}.

\begin{proof}[Proof of Theorem~\ref{thm:main}]
Let $\C^n$ carry the real Zariski topology. We begin by defining a certain subset $X \subset \C^n$ as the intersection of finitely many nonempty open sets. Since $\C^n \cong \R^{2n}$ is an irreducible topological space, $X$ will be nonempty as well. Furthermore, as any nonempty Zariski-open set, $X$ will be open and dense in the Euclidean topology. Afterwards, we prove that $G = \U(Gx)$ for all $x \in X$.

First, let $U \subset \C^n$ be the set of elements $x \in \C^n$ such that $Gx$ spans $\C^n$. This set is open, as it consists of those elements $x \in \C^n$ satisfying
\[ \det(A_1x, \dots, A_nx) \neq 0 \text{ for certain matrices } A_1, \dots, A_n \in G. \]
It is nonempty by the assumption that at least one orbit of $G$ spans $\C^n$.

Next, for any permutation $\pi \in \Sym(G)$ on $G$ with $\pi(I_n) = I_n$ we consider
\[ O_\pi = \{ x \in \C^n : \langle \pi(A)x, x \rangle \neq \langle Ax, x \rangle \text{ for some } A \in G \}, \]
where $\langle x,y \rangle = y^* x $ is the standard inner product on $\C^n$. As a union of open sets, $O_\pi$ is clearly open (at this point we actually need the \emph{real} Zariski topology). Note that $O_\pi$ may be empty, for example if $\pi$ is the identity permutation. We define $X$ as the intersection of $U$ and all nonempty $O_\pi$, where $\pi \in \Sym(G)$ with $\pi(I_n) = I_n$.

We proceed by showing that $G = \U(Gz)$ for an arbitrary element $z \in X$. Of course the inclusion $G \subseteq \U(Gz)$ is trivial. If $C \in \U(Gz)$, then $C$ permutes the elements of $Gz$ which means there is a permutation $\pi \in \Sym(G)$ such that
\begin{align}
CAz = \pi(A)z \text{ for all } A \in G. \label{eqn:C_permutes}
\end{align}
By multiplying $C$ by some element of $G$ if necessary, we may assume without loss of generality that $Cz=z$, and $\pi(I_n) = I_n$. As $C$ is unitary, we have
\[ \langle \pi(A)z, z \rangle = \langle CAz, Cz \rangle = \langle Az, z \rangle \text{ for all } A \in G, \]
so that $z \notin O_\pi$ by definition. We conclude that $O_\pi$ is empty, and hence 
\begin{align}
\langle \pi(A)v, v \rangle = \langle Av, v \rangle \label{eqn:un_orb_sym}
\end{align}
holds for all $A \in G$ and all $v \in \C^n$. By plugging $v = x+y$ into \eqref{eqn:un_orb_sym} for $x,y \in \C^n$ arbitrary and expanding both sides, we get
\begin{align*}
\langle \pi(A)x,x \rangle + \langle \pi(A)x,y \rangle + \langle \pi(A)y,x \rangle + \langle \pi(A)y,y \rangle\\
= \langle Ax,x \rangle + \langle Ax,y \rangle + \langle Ay,x \rangle + \langle Ay,y \rangle.
\end{align*}
By applying \eqref{eqn:un_orb_sym} again and canceling common terms, we get
\begin{align}
\langle \pi(A)x,y \rangle + \langle \pi(A)y,x \rangle = \langle Ax,y \rangle + \langle Ay,x \rangle. \label{eqn:un_orb_sym2}
\end{align}
In \eqref{eqn:un_orb_sym2} we replace $y$ by $iy$ and multiply both sides by $i$ to get
\begin{align}
\langle \pi(A)x,y \rangle - \langle \pi(A)y,x \rangle = \langle Ax,y \rangle - \langle Ay,x \rangle. \label{eqn:un_orb_sym3}
\end{align}
Combining \eqref{eqn:un_orb_sym2} and \eqref{eqn:un_orb_sym3} yields
\begin{align*}
\langle \pi(A)x,y \rangle = \langle Ax,y \rangle \text{ for all } A \in G.
\end{align*}
Since $x$ and $y$ were arbitrary, we conclude $\pi(A) = A$ for all $A \in G$.
Finally, we apply \eqref{eqn:C_permutes} again to obtain
\[ CAz = \pi(A)z = Az \text{ for all } A \in G. \]
Since $z \in U$ by definition, the orbit of $z$ spans $\C^n$ so that $C = I_n \in G$.
\end{proof}

\section*{Acknowledgement}
I thank Frieder Ladisch for many useful remarks which, in particular, helped to shorten the proof of Theorem~\ref{thm:main} significantly.

\end{document}